\documentclass[11pt,letterpaper]{article}
\usepackage[letterpaper,margin=1in]{geometry}
\usepackage[american]{babel}

\usepackage{dsfont}
\usepackage{amsthm}
\usepackage{amsmath}
\usepackage{amssymb} 
\usepackage{xspace}
\usepackage[usenames,dvipsnames]{xcolor}
\usepackage{enumerate}

\usepackage[textsize=tiny]{todonotes}

\usepackage{thmtools, thm-restate}
\declaretheorem[numberwithin=section]{theorem}

\declaretheorem[sibling=theorem]{lemma}

\declaretheorem[sibling=theorem]{conjecture}

\usepackage{multirow}

\usepackage[latin1]{inputenc}

\newcommand{\Ex}{\mathbb{E}}

\renewcommand{\epsilon}{\ensuremath{\varepsilon}}

\newcommand{\BE}{\ensuremath{\mathcal{B}}}

\begin{document}

\title{A new upper bound on the game chromatic index of graphs}
\author{Ralph Keusch \vspace{0.2cm}\\ \small Institute of Theoretical Computer Science\\
\small ETH Zurich, 8092 Zurich, Switzerland \vspace{0.2cm}\\ \small \texttt{rkeusch@inf.ethz.ch}}
\maketitle

\begin{abstract}
We study the two-player game where Maker and Breaker alternately color the edges of a given graph $G$ with $k$ colors such that adjacent edges never get the same color. Maker's goal is to play such that at the end of the game, all edges are colored. Vice-versa, Breaker wins as soon as there is an uncolored edge where every color is blocked. The game chromatic index $\chi'_g(G)$ denotes the smallest $k$ for which Maker has a winning strategy.

The trivial bounds $\Delta(G) \le \chi_g'(G) \le 2\Delta(G)-1$ hold for every graph $G$, where $\Delta(G)$ is the maximum degree of $G$. In 2008, Beveridge, Bohman, Frieze, and Pikhurko proved that for every $\delta>0$ there exists a constant $c>0$ such that $\chi'_g(G) \le (2-c)\Delta(G)$ holds for any graph with $\Delta(G) \ge (\frac12+\delta)v(G)$, and conjectured that the same holds for every graph $G$. In this paper, we show that $\chi'_g(G) \le (2-c)\Delta(G)$ is true for all graphs $G$ with $\Delta(G) \ge C \log v(G)$. In addition, we consider a biased version of the game where Breaker is allowed to color $b$ edges per turn and give bounds on the number of colors needed for Maker to win this biased game.
\end{abstract}

\section{Introduction}\label{sec:intro}

Let $G=(V,E)$ be a graph and let $k$ be a positive integer. We study the game where two players, called Maker and Breaker, take turns in which they alternately assign a color $i \in \{1, \ldots, k\}$ to a previously uncolored edge $e \in E$ such that the partial coloring stays proper, i.e., no two adjacent edges get the same color. Maker's goal is that at the end of the game, every edge is colored. Meanwhile, Breaker plays against Maker aims to produce a partial coloring such that for at least one uncolored edge, all colors are forbidden and thus the partial coloring can not be extended to a proper edge-coloring of $G$. The \emph{game chromatic index} $\chi'_g(G)$ is defined as the smallest integer $k$ for which Maker has a winning strategy.

This game is a variation of the analogous Maker-Breaker game where the players color vertices instead of edges. There, the \emph{game chromatic number} $\chi_g(G)$ denotes the smallest number of colors for which Maker has a winning strategy. The vertex coloring game is one of the classic Maker-Breaker games and well-understood by now. Recent results include forests \cite{faigle1993game}, planar graphs \cite{bartnicki2007map,kierstead1994planar,zhu2008refined}, and random graphs \cite{bohman2008game,frieze2013game,keusch2014game}.

The game chromatic index of graphs was first studied by Lam, Shiu, and Xu in 1999 \cite{lam1999edge}. For any graph $G$ we have the two trivial bounds
\begin{equation}\label{eq:trivial}
\Delta(G) \le \chi'_g(G) \le 2\Delta(G)-1,
\end{equation}
where $\Delta(G)$ denotes the maximum degree of $G$.
Clearly, the lower bound is tight for star graphs and the upper bound is tight for cycles of odd length. Cai and Zhu proved that $\chi'_g(G) \le \Delta + 3k-1$ holds for every $k$-degenerate graph $G$ \cite{cai2001game}. Erd\H{o}s, Faigle, Hochst\"{a}ttler, and Kern \cite{erdoes2004note} showed that for forests $T$ of maximum degree $\Delta(T) \ge 6$ it holds $\chi'_g(T) \le \Delta(T)+1$, and that for most forests this bound is tight. Afterwards, Andres extended this result to the case $\Delta=5$ \cite{andres2006game}. While the game chromatic index of forests differs only by a small constant from the maximum degree, it is known that for every sufficiently large $d$ there exists a graph $G$ with $\Delta(G) \le d$ and $\chi'_g(G) \ge 1.008d$ \cite{beveridge2008game}. Further results on the game chromatic index include graphs of bounded arboricity \cite{bartnicki2008game} and wheels \cite{lam1999edge}. 

Unfortunately, in general the game is believed to be hard to analyze. It seems challenging to find powerful strategies even for only one of the two players. For example, a player's move that looks clever at the start of the game can easily hurt the same player later on. That is why accurate bounds on the game chromatic index are only known for very few specific and sparse graph classes, and in general, knowledge on the game is rather scarce. Although it is desirable to determine $\chi'_g(G)$ precisely, it is reasonable to first decide whether $\chi'_g(G)$ is bounded away by a constant factor from $\Delta(G)$, from $2\Delta(G)$, or from both. In 2008, Beveridge, Bohman, Frieze, and Pikhurko \cite{beveridge2008game} proved that for every $\delta>0$ there exists  $c>0$ such that every graph $G$ of maximum degree at least $(1/2+\delta)v(G)$ satisfies $\chi'_g(G) \le (2-c)\Delta(G)$. Furthermore, they conjectured that the same is true for every graph $G$.

\begin{conjecture}[Conjecture~2 in \cite{beveridge2008game}]\label{conj:upper}
There exists $c>0$ such that for every graph $G$ it holds $$\chi'_g(G) \le (2-c)\Delta(G).$$
\end{conjecture}


\paragraph{Contribution of the paper.} Our main result provides a non-trivial upper bound on the game chromatic index for all graphs $G$ of maximum degree at least $C\log v(G)$. This extends the previous result of Beveridge et al.~\cite{beveridge2008game} and can be seen as a first step towards a proof of Conjecture~\ref{conj:upper}.

\begin{theorem}\label{thm:main}
There exist $C,c> 0$ such that any graph $G$ with $\Delta(G) \ge C\cdot\log v(G)$ satisfies
$$\chi_g'(G) \le (2-c)\Delta(G).$$
\end{theorem}

Note that in particular for full bipartite graphs or random graphs $G(n,p)$ with appropriate parameter $p$, Theorem~\ref{thm:main} yields the first non-trivial bound on the game chromatic index. The result also generalizes to the variant of the game where
Breaker is allowed to sit out during his turns. Consequently, the identity of the starting player does not matter. 

In the context of Maker-Breaker games, it is natural to also consider biased games where one player is allowed to claim not only one but multiple elements per round. Let $b \ge 1$. We introduce the edge coloring game with bias $b$ as follows. Maker still colors a single edge per round as before, but in each of his turns, Breaker is now allowed to color any number of edges that is at most $b$. The winning conditions for the two players remain the same. For this biased variant of the game, we define $\chi'_g(G,b)$ as the smallest number of colors such that Maker has a winning strategy. Clearly, the bounds of \eqref{eq:trivial} are still valid. We show that Theorem~\ref{thm:main} can be generalized to the biased edge coloring game.

\begin{theorem}\label{thm:upperbiased}
There exists $c>0$ such that for any $b \ge 1$, any graph $G$ with $\Delta(G) \ge C(b) \cdot \log v(G)$ satisfies
$$\chi_g'(G,b) \le (2-cb^{-4})\Delta(G),$$
where $C(b)>0$ only depends on $b$.
\end{theorem}

In strong contrast, our last result verifies that there are graphs $G$ where a bias $b\ge 2$ results in Breaker winning the game even with $2\Delta(G)-2$ colors, and thus $\chi'_g(G,b)=2\Delta(G)-1$. Hence, an analogue of Conjecture~\ref{conj:upper} can not hold for the biased variant of the game. In particular, for regular graphs it follows that the precondition of Theorem~\ref{thm:upperbiased} is almost optimal and that the value of $\chi'_g(G,b)$ depends not only on the maximum degree but also on the number of vertices.

\begin{theorem}\label{thm:lowerbiased}
There exists $C>0$ such that for all $b\ge 2$ and $\Delta \ge 2$, every $\Delta$-regular graph $G$ with at least $C\cdot\Delta^3 \cdot\exp(\frac{\Delta-1}{b-1})$ vertices satisfies
$$\chi'_g(G,b)=2\Delta(G)-1.$$
\end{theorem}

\paragraph{Organization.} After introducing some notations in Section~\ref{sub:notations}, we present in Section~\ref{sub:Mstrategy} a randomized strategy for Maker which we then analyze in Section~\ref{sub:mainproof} in order to prove Theorem~\ref{thm:main} and Theorem~\ref{thm:upperbiased}. Afterwards, in Section~\ref{sec:lowerbound} we give a short proof of Theorem~\ref{thm:lowerbiased} by constructing a reduction of the biased edge coloring game to so-called \emph{Box games}. We conclude with a brief discussion of several open problems in Section~\ref{sec:openproblems}.

\section{Upper bounds}\label{sec:upperbounds}

\subsection{Notations}\label{sub:notations}
We start with some notations. We consider the game as a process that evolves in rounds. In the first round, only Breaker is allowed to play. Afterwards, in every round $r$ it is first Maker's and then Breaker's turn. When proving upper bounds, we allow Breaker to sit out and not color any edge in his turns. This setting was first studied by Andres \cite{andres2006game} and makes the identity of the starting player irrelevant. Clearly, any upper bound on $\chi'_g(G)$ that holds for this modified variant serves also an upper bound for the original game. 
We abbreviate $\Delta=\Delta(G)$, denote by $\Gamma(v)$ the set of neighbors of a vertex $v$, and by $v$-edge an edge that is incident to $v$. $\Gamma'_r(v)$ is defined as the set of all neighbors of $v$ in the subgraph of uncolored edges after round $r$. Furthermore, the \emph{load} $\ell_r(v) := \deg(v)-|\Gamma'_r(v)|$ counts the total number of colored $v$-edges after round $r$. Finally, let $A_r(e)$ be the set of available colors at an edge $e$ after round $r$, and let $U_r(v)$ be the set of colors that have been used at $v$-edges during the first $r$ rounds.

%

\subsection{Maker's strategy}\label{sub:Mstrategy}

We will prove Theorem~\ref{thm:main} and Theorem~\ref{thm:upperbiased} by providing a random strategy for Maker. Note that as we study a complete information game without chance moves, there exists a winning strategy for exactly one of the two players. Hence, it is sufficient to prove that the proposed random strategy of Maker wins with strictly positive probability against any fixed, deterministic strategy of Breaker. Then Breaker can not have a winning strategy, implying that there exists a deterministic winning strategy for Maker. This application of the probabilistic method was first used by Spencer \cite{spencer1991randomization}.
 
Let $G=(V,E)$ be a given graph. We fix $\lambda$ and $c$ globally such that
$$1 \gg \lambda \gg c > 0.$$
For the sake of readability, we always ignore roundings and assume that all considered quantities are integers. We will prove Theorem~\ref{thm:main} and Theorem~\ref{thm:upperbiased} at once and therefore assume that there exists a fixed integer $b \ge 1$ such that in each of his turns, Breaker colors at least $0$ and at most $b$ edges. Note that $\lambda$ and $c$ do not depend on $b$. Given this set of constants, we assume that the game is played with a set of $k :=(2-c b^{-4})\Delta$ colors.

Before defining the strategy, we make some further preparations. Suppose an uncolored edge $e=\{u,v\}$ satisfies $|U_r(u) \cap U_r(v)| > 2\Delta-k = cb^{-4}\Delta$ after some round $r$. Then the edge $e$ will never run out of available colors. For every vertex $v \in V$, Maker uses this observation as follows. After the first round $r$ where $\ell_r(v) \ge 2\lambda b^{-1} \Delta$ holds, he looks at the set of uncolored $v$-edges, and defines a set $D(v)\subseteq \Gamma(v)$ of \emph{dangerous} neighbors, containing all vertices $u \in \Gamma(v)$ that fulfill the following four conditions:
\begin{enumerate}[(i)]
\item the edge $\{u,v\}$ is still uncolored, i.e., $u \in \Gamma'_r(v)$,
\item $\deg(u)+\deg(v)\ge k=(2-cb^{-4})\Delta$,
\item $|U_r(u) \cap U_r(v)| \le 2\Delta-k=cb^{-4}\Delta$, and
\item $u$ reached load $\lambda b^{-1}\Delta$ not after its neighbor $v$, i.e., for all $r'$ such that $\ell_{r'}(v) \ge \lambda b^{-1}\Delta$ it also holds $\ell_{r'}(u) \ge \lambda b^{-1}\Delta$.
\end{enumerate}

Clearly, an edge $\{u,v\}$ can run out of available colors only if (i)-(iii) are fulfilled. Intuitively speaking, with condition~$(iv)$ we decide which vertex is responsible for such an edge. In case $u$ and $v$ reach load $\lambda b^{-1} \Delta$ at the same round, the construction yields $u \in D(v)$ and $v \in D(u)$. Once we are at a round $r$ such that $v$ satisfies $\ell_r(v) \ge 2\lambda b^{-1}\Delta$, the set $D(v)$ is defined and Maker's local goal for the remaining game process will be to color all edges between $v$ and $D(v)$ before $v$ reaches load $3\lambda b^{-1}\Delta$. Note that as long as $\ell_r(v) \le 3\lambda b^{-1}\Delta$, there are still colors available for these edges because $\lambda$ is chosen sufficiently small.

We now start describing Maker's strategy at an arbitrary round $r$ where there are still uncolored edges left. Let $f_0$ be the edge colored by Maker at round $r-1$. (If it is Maker's first move of the game, take an arbitrary edge for $f_0$.) Furthermore, let $F$ be the set of edges that Breaker colored in his turn at round $r-1$. In the special case where Breaker didn't color any edge in his last move, choose instead any uncolored edge $f_1$ and put $F := \{f_1\}$. Let $q := \frac{6c}{\lambda}$. We then propose Maker to play at random in the following way.

\begin{enumerate}
\item Choose $f \in \{f_0\} \cup F$ at random such that $\Pr[f=f_0] = \frac12$ and $\Pr[f=f_i] = \frac{1}{2|F|}$ for all $f_i \in F$.
\item Let $v$ be one of the two vertices incident to $f$ chosen uniformly at random. If $\Gamma'_{r-1}(v)$ is empty, replace $v$ with another vertex $v$ such that $\Gamma'_{r-1}(v)$ is non-empty.
\item Choose a neighbor $u \in \Gamma'_{r-1}(v)$ uniformly at random. If $\ell(v) \ge 2\lambda b^{-1}\Delta$ and $D(v) \cap \Gamma'_{r-1}(v)$ is non-empty, with probability $q$ discard the first choice of $u$ and replace it by $u \in D(v) \cap \Gamma'_{r-1}(v)$ chosen uniformly at random.
\item Let $e = \{u,v\}$ and color $e$ with a color $i \in A_{r-1}(e)$ chosen uniformly at random. We call $e$ a \emph{good} v-edge.
\end{enumerate}

Note that the strategy is well-defined, i.e., it always yields an uncolored edge $e$ that Maker has to color. For every edge $f_i \in \{f_0\} \cup F$, the probability that it is chosen by Maker in the first step is at least $\frac{1}{2b}$, as we have $|F| \le b$ by assumption. 

Suppose Maker applies the proposed strategy. In case the strategy tells Maker to color an edge $e \in E$ at round $r$ but $A_{r-1}(e)$ is empty, Maker loses the game by definition. Then, we don't yet abort the game. Instead, we let Maker play a color $i$ chosen uniformly at random among \emph{all} colors and create a non-proper coloring, whereas Breaker is still forced to color edges properly. Consequently, if there is a left-over of uncolored edges where all colors are blocked, Breaker has no other option than sitting out for the remainder of the game (which is indeed possible for him). This yields a slightly different coloring process that always terminates with a full but not necessarily proper edge coloring. Observe that as long as Maker never needs to use forbidden colors in the modified process, the original and the modified process coincide. If Maker is never forced to use forbidden colors, then in both processes we obtain a proper edge coloring of $G$ and Maker wins the game.

\subsection{Main proof}\label{sub:mainproof}

Since Theorem~\ref{thm:main} is a special case of Theorem~\ref{thm:upperbiased}, it suffices to prove the latter. 
By the precondition of Theorem~\ref{thm:upperbiased}, the maximum degree $\Delta$ is larger than $C(b)$, so we can always assume that $\Delta$ is sufficiently large. 

We start the analysis by collecting several auxiliary results. Let $v \in V$ be any vertex. Recall the definitions of $v$-edges and good $v$-edges. Our first goal is to verify that while the load $\ell_r(v)$ increases during the game process, always a constant fraction of the colored $v$-edges are \emph{good} $v$-edges. We specify this with the following lemma.


\begin{lemma}\label{lem:rates}
Let $j \in \{1,2,3\}$ and let $v \in V$ be a vertex of degree at least $j\lambda b^{-1} \Delta$. Denote by $\BE_j(v)$ the bad event that among the $v$-edges that have been colored at rounds $r$ where $\ell_{r-1}(v) \ge (j-1)\lambda b^{-1}\Delta$ and $\ell_r(v) < j\lambda b^{-1}\Delta$, less than $\frac{1}{5b^2}\lambda\Delta$ edges are good $v$-edges. Then
$$\Pr[\BE_j(v)]=\exp(-\Omega(\Delta)).$$
\end{lemma}

We defer the proof of Lemma~\ref{lem:rates} together with the proofs of the two subsequent lemmas to Section~\ref{sub:missingproofs}. Note that here and in the following, whenever we use the Landau-notation for probability estimations, we hide constant factors that may depend on $\lambda$, $c$, or $b$.

Next, we study how fast the load of a vertex of large degree grows compared to the average load of its neighbors. We show that it is unlikely that the average load among vertices in $\Gamma'_r(v)$ deviates by more than a constant factor from $\ell_r(v)$.

\begin{lemma}\label{lem:nbload}
Let $v \in V$ be a vertex of degree at least $k-\Delta=(1-c b^{-4})\Delta$. Denote by $\BE_4(v)$ the bad event that there exists a round $r$ where $\ell_r(v)< 2\lambda b^{-1} \Delta$ but $\frac{1}{|\Gamma'_r(v)|}\sum_{u \in \Gamma'_r(v)} \ell_r(u) \ge 9 \lambda\Delta$. Then
$$\Pr[\BE_4(v)]=\exp(-\Omega(\Delta)).$$
\end{lemma}

As a next step, we also take colors into consideration. Let $v$ be a vertex of degree at least $\lambda b^{-1}\Delta$. Then we know that unless the bad event $\BE_1(v)$ occurs, among the first $\lambda b^{-1}\Delta$ colors used at $v$-edges, there are at least $\frac{1}{5b^2}\lambda\Delta$ colors that were assigned by Maker to good $v$-edges. As long as $\ell_r(v) < \lambda b^{-1} \Delta$, for all $v$-edges the set $A_r(e)$ is non-empty, Maker is not forced to color $v$-edges non-properly, and indeed all colored $v$-edges use distinct colors. 

For any vertex $v$, let $I'(v)$ be a subset of colors assigned to good $v$-edges defined as follows. If Maker colors less than $\frac{1}{5b^2}\lambda\Delta$ good $v$-edges before the load of $v$ reaches $\lambda b^{-1}\Delta$ (i.e., at rounds $r$ such that $\ell_r(v) < \lambda b^{-1}\Delta$), then $I'(v)$ is the set of colors that Maker used for these good $v$-edges. If there are at least $\frac{1}{5b^2}\lambda\Delta$ such edges, $I'(v)$ only contains the first $\frac{1}{5b^2}\lambda\Delta$ colors that Maker used at such moves.
For a vertex $v$, we hope that colors are distributed rather randomly inside the sets $\{I'(u):u \in \Gamma(v)\}$. We formalize such a distribution with the following lemma. 

\begin{lemma}\label{lem:colormixing}
Let $v \in V$ be any vertex. Denote by $\BE_5(v)$ the event that there exists a subset of neighbors $W \subseteq \Gamma(v)$ of size $c b^{-2} \Delta$ and a set $I^-$ of $c b^{-2}\Delta$ colors such that for all $i \in I'$, we have $|\{u \in W: i \in I'(u)\}| \ge \frac{1}{4b^4}c\lambda\Delta$. Then
$$\Pr[\BE_5(v)] = \exp(-\Omega(\Delta^2)).$$
\end{lemma}

We now start proving the main theorem.

\begin{proof}[Proof of Theorem~\ref{thm:upperbiased}]
By definition of the sets $D(v)$, it is sufficient to verify for all vertices $v$ of degree at least $k-\Delta=(1-cb^{-4})\Delta$ that either $\deg(v) \le 3\lambda b^{-1} \Delta$ (and thus incident edges never run out of available colors) or that Maker is fast enough to color all edges between $v$ and $D(v)$ before $v$ reaches load $3\lambda b^{-1}\Delta$. If this is possible for Maker, then the process yields a proper edge coloring of $G$ and Maker wins the game.

Let $v \in V$ be a fixed vertex of degree at least $(1-c b^{-4})\Delta$. Denote by $D'(v) \subseteq \Gamma(v)$ the set of all neighbors of $v$ that reach load $\lambda b^{-1}\Delta$ not after $v$. More precisely, $D'(v)$ contains all neighbors $u \in \Gamma(v)$ such that $\ell_r(v)\ge \lambda b^{-1}\Delta$ always implies $\ell_r(u)\ge \lambda b^{-1}\Delta$. 
In the following we assume that for all $u \in D'(v)$ the bad event $\BE_1(u)$ and for $v$ itself the bad event $\BE_5(v)$ do not occur. By Lemma~\ref{lem:rates}, Lemma~\ref{lem:colormixing}, and a union bound, this happens with probability $1-\exp(-\Omega(\Delta))$.

Suppose $|D'(v)| \ge c b^{-2}\Delta$ and define $\mathcal{W} := \{W \subseteq D'(v): |W| =c b^{-2}\Delta\}$. Let us first consider a fixed set $W \in \mathcal{W}$. As we are excluding the events $\{\BE_1(u):u \in W\}$ and $\BE_5(v)$, for every vertex $u \in W$ there exists a set $I'(u)$ of exactly $\frac{1}{5b^2} \lambda\Delta$ colors with the property that the number of colors $i$ satisfying $|\{u \in W: i \in I'(u)\}|\ge \frac{1}{4b^4}c\lambda\Delta$ is at most $c b^{-2}\Delta$. Next, let us define $I_W$ as the set of all colors $i$ that fulfill $|\{u \in W: i \in I'(u)\}|\ge \frac{1}{24b^4}c \lambda\Delta$. We claim that $|I_W| \ge \frac{\Delta}{2}$. Indeed, if this is not the case, then there are at least $k-\frac{\Delta}{2}=(\frac32-cb^{-4})\Delta$ colors that are contained in at most $\frac{1}{24b^4}c\lambda\Delta$ of the sets $\{I'(u): u \in W\}$. On the other hand, except a small set of at most $c b^{-2}\Delta$ heavy colors, all colors $i \in I_W$ are contained in at most $\frac{1}{4b^4}c\lambda\Delta$ of the sets $\{I'(u):u \in W\}$. If $c$ is sufficiently small compared to $\lambda$, this yields
\begin{align*}
\sum_{i=1}^k | \{u \in W: i \in I'(u)\} &\le \Big( \frac32-cb^{-4}\Big)\Delta \cdot \frac{1}{24b^4} c\lambda\Delta + cb^{-2}\Delta \cdot cb^{-2}\Delta + \Big(\frac12-cb^{-2}\Big)\Delta \cdot \frac{1}{4b^4}c\lambda\Delta\\
& \le \frac{3}{16b^4} c\lambda\Delta^2 + \frac{1}{b^4}c^2 \Delta^2\\
& < \frac{1}{5b^4}c\lambda\Delta^2\\
& = \sum_{u \in W} |I'(u)|,
\end{align*}
which is clearly a contradiction as the first and last term of the inequality chain are equal.

Now let us look at the period of the game process that contains all rounds $r$ where the load of $v$ fulfills $\ell_{r-1}(v) \ge \lambda b^{-1}\Delta$ and $\ell_r(v) < 2 \lambda b^{-1} \Delta$. Denote by $I_v$ the set of colors assigned to good $v$-edges within this period and by $\BE_6(v)$ the bad event that there exists a set $W \in \mathcal{W}$ such that $|I_v \cap I_W| < \frac{1}{100b^2}\lambda\Delta$. In the following, we show that $\Pr[\BE_6(v)] = \exp(-\Omega(\Delta))$. 

Consider a single round $r$ within this period and condition on that at round $r$, Maker colors a good $v$-edge $e=\{v,w\}$ with color $i$. Then $i$ is added to the set $I_v$. Suppose that we have 
$$\sum_{u \in \Gamma'_{r-1}(v)} \ell_{r-1}(u) \le 9\lambda \Delta \cdot|\Gamma'_{r-1}(v)|  \le 9\lambda \Delta^2.$$
In this case, we infer from Markov's inequality that there exist at most $\frac{9}{14}\Delta$ vertices $u \in \Gamma'_{r-1}(v)$ with the property $\ell(u) \ge 14\lambda  \Delta$. Moreover, 
$$|\Gamma'_{r-1}(v)| \ge \deg(v)-2\lambda b^{-1}\Delta \ge (1-cb^{-4}-2\lambda b^{-1})\Delta \ge \frac{27}{28} \Delta,$$ 
if $c$ and $\lambda$ are chosen sufficiently small. Hence at round $r$, the vertex $w$ that Maker chooses for his edge $\{v,w\}$ uniformly at random in $\Gamma'_{r-1}(v)$ satisfies $\ell_{r-1}(w) \le 14 \lambda  \Delta$ with probability at least $\frac13$. If the random choice yields such a neighbor $w$, it also follows 
$$|A_{r-1}(e)| \ge k-\ell_{r-1}(v)-\ell_{r-1}(w) \ge (2-cb^{-4}-16\lambda)\Delta.$$
Since Maker takes $i \in A(e)$ uniformly at random and $\lambda$ is chosen sufficiently small, we have
$$\Pr[i \in I_W] \ge 1-\frac{k-I_W}{|A_{r-1}(e)|} \ge 1 - \frac{3/2-cb^{-4}}{2-cb^{-4}-16\lambda} = \frac{1-32\lambda}{4-2cb^{-4}-32\lambda} \ge \frac15.$$
We summarize that as long as we have $\sum_{u \in \Gamma'_{r-1}(v)} \ell_{r-1}(u) \le 9\lambda \Delta \cdot|\Gamma'_{r-1}(v)|$, a color $i$ that is added to $I_v$ at round $r$ is also contained in $I_W$ with probability at least $\frac13 \cdot \frac15 = \frac{1}{15}$, independently of the success in previous rounds as we do not yet condition on any good or bad events concerning the actual time period. 

Let $m := \frac{1}{5b^2}\lambda\Delta$ and let $(X^W_1, X^W_2, \ldots)$ be an infinite $0/1$-sequence where each entry is $1$ independently with probability $\frac{1}{15}$. 
We use the sequence $(X^W_i)_{i \ge 1}$ for a coupling as follows. Whenever Maker adds a color to $I_v$ at a round $r$ and  $\sum_{u \in \Gamma'_{r-1}(v)} \ell_{r-1}(u) \le |\Gamma'_{r-1}(v)| \cdot 9\lambda \Delta$, we read the next bit $X^W_i$ of $(X^W_i)_{i \ge 1}$. Then the coupling is such that $X^W_i=1$ implies $i \in I_W$. Clearly $\mu := \Ex[\sum_{i=1}^m X^W_i] = \frac{m}{15}=\frac{1}{75b^2}\lambda\Delta$, and by a Chernoff bound we deduce
$$\Pr\Big[\sum_{i=1}^m X^W_i \le \frac{1}{100b^2}\lambda\Delta\Big] \le \Pr\Big[\sum_{i=1}^m X^W_i \le \frac{3}{4}\mu\Big] = \exp\Big(-\frac{\mu}{4^2\cdot 2}\Big)=\exp\Big(-\frac{\lambda\Delta}{16\cdot 2 \cdot 75b^2}\Big).$$
Next, we do a union bound over all sets $W \in \mathcal{W}$. Using the inequality $\binom{n}{k} \le (\frac{ne}{k})^k$ we obtain
\begin{align*}
\Pr\Big[\bigwedge_{W \in \mathcal{W}} \Big\{\sum_{i=1}^m X^W_i \le \frac{1}{100b^2}\lambda\Delta \Big\} \Big] &\le \binom{\Delta}{c b^{-2}\Delta} \cdot \exp\Big(-\frac{\lambda\Delta}{16\cdot 2 \cdot 75b^2}\Big) \\
& \le \Big(\frac{e}{cb^{-2}}\Big)^{cb^{-2}\Delta} \cdot \exp\Big(-\frac{\lambda\Delta}{2400b^2}\Big)\\
& = \exp(-\Omega(\Delta)),
\end{align*}
where the last step follows if $c$ is chosen sufficiently small compared to $\lambda$.

By Lemma~\ref{lem:rates} and Lemma~\ref{lem:nbload} we have $\Pr[\BE_2(v) \cup \BE_4(v)] = \exp(-\Omega(\Delta))$. Hence with probability $1-\exp(-\Omega(\Delta))$ it holds $|I_v| \ge \frac{1}{5b^2}\lambda\Delta$, and as long as $\ell_r(v) < 2\lambda b^{-1}\Delta$, we also have 
$$\sum_{u \in \Gamma'_r(v)} \ell_r(u) < 9\lambda \Delta \cdot |\Gamma'(v)|.$$ 
In this case, for all $W \in \mathcal{W}$ the size of the set $I_v \cap I_W$ is lower-bounded by $\sum_{i=1}^m X^W_i$. By a union bound over all bad events, it follows that with probability $1-\exp(-\Omega(\Delta))$, all $W \in \mathcal{W}$ satisfy
$$|I_v \cap I_W| \ge \frac{1}{100b^2}\lambda\Delta.$$
Therefore,
$$\Pr[\BE_6(v)] = \exp(-\Omega(\Delta)).$$

Suppose now that the bad event $\BE_6(v)$ does not happen. Then for every set $W \in \mathcal{W}$ we have
$$\sum_{u \in W} |I'(u) \cap I_v| \ge \sum_{u \in W} |I'(u) \cap I_v \cap I_W| \ge \frac{1}{100b^2}\lambda\Delta \cdot \frac{1}{24b^4} c\lambda\Delta = \frac{1}{2400b^6}c\lambda^2\Delta^2.$$
Let $s := \min\{r: \ell_r(v) \ge 2\lambda b^{-1}\Delta\}$. By an averaging argument, we see that for every $W \in \mathcal{W}$ there exists a vertex $u_W \in W$ such that
$$|U_s(u_W) \cap U_s(v)| \ge |I'(u_W) \cap I_v| \ge \frac{1}{2400b^4} \lambda^2 \Delta > cb^{-4} \Delta,$$
given that $c$ sufficiently small compared to $\lambda$. Hence, for every set $W \in \mathcal{W}$ there exists at least one vertex that does not belong to $D(v)$. Because $D(v) \subseteq D'(v)$, it follows
\begin{equation}\label{eq:boundbadnbs}
|D(v)| \le c b^{-2}\Delta.
\end{equation}

Once having derived \eqref{eq:boundbadnbs}, we proceed by considering the rounds $r$ where $\ell_{r-1}(v) \ge 2\lambda b^{-1}\Delta$ and $\ell_r(v) < 3\lambda b^{-1}\Delta$. We want to show that within this period, Maker is fast enough to color all uncolored edges between $v$ and $D(v)$. Recall from Maker's strategy that whenever he colors a good $v$-edge at a round $r$ and $D(v) \cap \Gamma'_{r-1}(v)$ is non-empty, with probability at least $q=\frac{6 c}{\lambda}$ Maker chooses a vertex $w \in D(v)$ and colors the edge $\{v,w\}$. Again, we couple the process with an infinite $0/1$-sequence $(X_1, X_2, \ldots)$ where each entry is $1$ independently with probability $q$. Whenever $D(v) \cap \Gamma'_{r-1}(v)$ is non-empty and Maker is about to color a good $v$-edge at round $r$, we read the next bit $X_i$ of $(X_i)_{i \ge 1}$. If $X_i=1$, we require that  $w \in D(v)$. Recall that $m = \frac{1}{5b^2}\lambda\Delta$. Clearly $\mu' := \Ex[\sum_{i=1}^{m} X_i] = m \cdot q \ge \frac65 c b^{-2}\Delta$. Denote by $\BE_7(v)$ the bad event that $\sum_{i=1}^{m} X_i \le cb^{-2}\Delta$. By a Chernoff bound it holds
$$\Pr[\BE_7(v)] \le \Pr\Big[\sum_{i=1}^{m} X_i \le \frac56 \mu'\Big]=\exp(-\Omega(\mu')) = \exp(-\Omega(\Delta)).$$

Assume $B_3(v)$ and $B_7(v)$ do not occur. Then in the considered period of the process where $\ell_r(v)$ increases from $2\lambda\Delta$ to $3\lambda\Delta$, Maker colors at least $m'$ good $v$-edges, implying that either $|D(v)| > c b^{-2}\Delta$ (which contradicts \eqref{eq:boundbadnbs}) or Maker is fast enough and colors all edges between $v$ and $D(v)$ before the load of $v$ is above $3\lambda b^{-1}\Delta$. Hence, for all rounds $r$ such that $\ell_r(v) \ge 3\lambda b^{-1}\Delta$ we have $D(v)\cap \Gamma'_{r-1}(v)=\emptyset$, implying that indeed all $v$-edges can be colored \emph{properly}. 

We see that as long as for all $v \in V$ no bad event $\BE_j(v)$ happens, Maker is never forced to use forbidden colors, meaning that the process yields a proper coloring of the complete edge set $E$. Recall that we are assuming $\Delta(G) \ge C(b)\log v(G)$. Then by a union bound we have
$$\Pr\Big[ \bigvee_{v \in V} \bigvee_{j=1}^7 \BE_j(v)\Big] \le n \cdot \exp(-\Omega(\Delta)) =\exp(-\Omega(\Delta)) < 1$$
for $C$ sufficiently large. We conclude that (a) Maker wins with probability $1-\exp(-\Omega(\Delta))$ when applying the proposed strategy and (b) Maker \emph{has} a deterministic winning strategy. This finishes the main proof.
\end{proof}

\subsection{Missing proofs}\label{sub:missingproofs}

\begin{proof}[Proof of Lemma~\ref{lem:rates}]
Let $j \in \{1,2,3\}$ and let $v \in V$ be a vertex of degree at least $j\lambda\Delta$. We define $R = \{r_1, \ldots, r_{|R|}\}$ as the set of rounds satisfying $\ell_{r-1}(v) \ge (j-1)\lambda b^{-1}\Delta$ and $\ell_r(v) < j\lambda  b^{-1}\Delta-(b+1)$ in which $v$-edges get colored by any of the two players. Maker's strategy is such that after every round $r_i \in R$, with non-zero probability Maker colors a good $v$-edge at round $r_i+1$. Let $X(i)$ be the indicator random variable for this event. Furthermore, for $1 \le i \le |R|$ let $e(i) \in \{1, \ldots, b+1\}$ be the number of $v$-edges that have been colored by Maker and Breaker at round $r_i$. Note that the values $e(i)$ depend on Breaker's strategy, which may itself heavily depend on Maker's random answers in previous moves as Breaker might apply an adaptive strategy. By definition of Maker's strategy, for all $i \in \{1, \ldots, |R|\}$ we independently have
\begin{equation}\label{eq:vgoodbound}
\Pr[X(i)=1] \ge \frac{e(i)}{4b}.
\end{equation}

For all $1 \le b' \le b+1$ let $(Y^{b'}_1, Y^{b'}_2, \ldots)$ be an infinite $0/1$-sequence where each entry is $1$ independently with probability $\frac{b'}{4b}$. We use this set of $0/1$-sequences for a coupling as follows. Whenever there is a new round $r_i \in R$, we read the next entry $Y^{e(i)}_j$ of the sequence $(Y^{e(i)}_j)_{j \ge 1}$. We require that $Y^{e(i)}_j=1$ implies $X(i)=1$, i.e., Maker plays a good $v$-edge at round $r_i+1$. By \eqref{eq:vgoodbound}, this is a valid coupling. Let $1 \le b' \le b+1$ and $1 \le m \le \lambda b^{-1}\Delta$. We have $\Ex[\sum_{j=1}^m Y^{b'}_j]=\frac{b'm}{4b}$, and by a Chernoff bound, 
\begin{align*}
\Pr\Big[\sum_{j=1}^m Y^{b'}_j < \frac{b'm}{4b} - \frac{\lambda\Delta}{25b^2(b+1)}\Big] &\le \Pr\Big[\sum_{j=1}^m Y^{b'}_j < \Big(1-\frac{4\lambda\Delta}{25(b'm)b(b+1)}\Big)\frac{b'm}{4b} \Big]\\
&\le \exp\Big(-\Omega\Big(\frac{\Delta^2}{m}\Big)\Big) \le \exp(-\Omega(\Delta)).
\end{align*}
By a union bound, with probability $1-\exp(-\Omega(\Delta))$, for all choices of $b'$ and $m$ it holds simultaneously
$$\sum_{j=1}^m Y^{b'}_j \ge \frac{b'm}{4b}-\frac{\lambda\Delta}{25b^2(b+1)}.$$

Suppose this good event happens. For all $1 \le b' \le b+1$, denote by $\alpha(b')$ the total number of rounds $r_i \in R$ such that $e(i)=b'$. Clearly, the random variables $\alpha(b')$ are upper-bounded by $\lambda b^{-1}\Delta$. Moreover, $\ell_{r_1}(v) \le (j-1)\lambda b^{-1} \Delta + (b+1)$ and $\ell_{r_{|R|}}(v) \ge j \lambda b^{-1}\Delta - 2(b+1)$, so
$$\sum_{b'=1}^{b+1} b' \cdot \alpha(b') \ge \lambda b^{-1}\Delta-3(b+1).$$
No matter how Breaker plays, it follows
$$
\sum_{i=1}^{|R|} X(i) \ge \sum_{b'=1}^{b+1} \sum_{j=1}^{\alpha(b')}Y^{b'}_j \ge \sum_{b'=1}^{b+1} \Big(\frac{b' \alpha(b')}{4b} - \frac{\lambda\Delta}{25b^2(b+1)}\Big) \ge \frac{\lambda\Delta}{4b^2}-2-\frac{\lambda\Delta}{25b^2} \ge \frac{\lambda\Delta}{5b^2}.
$$
However, by construction $\ell_{r_i+1}(v) < j \lambda b^{-1} \Delta$ holds for all $r_i \in R$. Hence, $\sum_{i=1}^{|R|} X(i)$ lower-bounds the total number of good $v$-edges in the considered period of the process, and with probability $1-\exp(-\Omega(\Delta))$, Maker is sufficiently fast in coloring good $v$-edges.
\end{proof}


\begin{proof}[Proof of Lemma~\ref{lem:nbload}]

Let $v \in V$ be a vertex of degree at least $(1-c b^{-4})\Delta$. We study how fast $\ell_r(v)$ grows compared to $L_r := \sum_{u \in \Gamma'_r(v)} \ell_r(u)$ over time. First, note that whenever a player colors an edge $\{u,v\}$ at round $r$, we have $u \notin \Gamma'_r(v)$, and thus the edge $\{u,v\}$ does not contribute to $L_r$. Therefore, an edge $e=\{u,w\}$ that is played at round $r$ and contributes to $L_r$ is either (1) such that $u \in \Gamma'_r(v)$ and $w \notin \Gamma'_r(v) \cup \{v\}$, or (2) such that $u,w \in \Gamma'_r(v)$. For a single edge $e$ of type (1), due to the proposed strategy, with probability at least $\frac{1}{4b}$ Maker answers by coloring a good $u$-edge in his next move at round $r+1$, no matter whether $e$ was colored by Maker or Breaker. In this case, with probability at least $\frac{1-q}{\Delta}$ he colors the edge $\{u,v\}$. All together, for an edge of type (1) the probability that Maker's next edge at round $r+1$ increases the load of $v$ is at least $\frac{1-q}{4b\Delta}$. For an edge of type (2), the same argument yields that with probability at least $\frac{1-q}{2b\Delta}$, Maker answers by coloring a $v$-edge at round $r+1$.

Let $R =\{r_1, \ldots, r_{|R|}\}$ be the set of rounds in which at least one edge of type (1) or (2) is played. Note that $|R| \le \Delta^2$ is a random variable. For all $1 \le i \le |R|$ let $e_1(i)$ be the number of edges of type (1) played at round $r_i$, let $e_2(i)$ be the same for edges of type (2), and put $e(i):=e_1(i)+2e_2(i) \in \{1, \ldots, 2b+2\}$. Let $X(i)$ be the indicator random variable for the event that at round $r_i+1$, Maker colors a $v$-edge. Then for all $1 \le j \le |R|$ the sum $\sum_{i=1}^j X(i)$ lower-bounds $\ell_{r_j+1}$, and by the previous observations we know that
\begin{equation}\label{eq:vnewbound}
\Pr[X(i)=1] \ge \frac{(1-q)e(i)}{4b\Delta}.
\end{equation}

For all $1 \le b' \le 2b+2$ let $(Y^{b'}_1, Y^{b'}_2, \ldots)$ be an infinite $0/1$-sequence where each entry is $1$ independently with probability $\frac{(1-q)b'}{4b\Delta}$. We build a coupling by using this set of $0/1$-sequences. After every round $r_i \in R$, we read the next bit $Y_j^{e(i)}$ of the sequence $(Y^{e(i)}_j)_{j \ge 1}$ and require that whenever $Y^{e(i)}_j$ equals one, then at the next round $r_i+1$, Maker plays a $v$-edge, implying $X(i)=1$. By \eqref{eq:vnewbound}, indeed this coupling is possible. Let $1 \le b' \le 2b+2$ and $1 \le m \le \frac{17}{2}\lambda \Delta^2$. We have $\Ex[\sum_{j=1}^m Y^{b'}_j]=\frac{(1-q)b'm}{4b\Delta}$, and by a Chernoff bound
\begin{align*}
\Pr\Big[\sum_{j=1}^m Y^{b'}_j < \frac{(1-q)b'm}{4b\Delta}-\frac{\lambda\Delta}{32b(b+1)}\Big] & \le \Pr\Big[\sum_{j=1}^m Y^{b'}_j < \Big(1 - \frac{\lambda\Delta^2}{8(1-q)(b'm)(b+1)}\Big)\frac{(1-q)b'm}{4b\Delta}\Big]\\
& \le \exp\Big(-\Omega\Big(\frac{\Delta^3}{m}\Big)\Big) \le \exp(-\Omega(\Delta)).
\end{align*}
By a union bound, with probability $1-\exp(-\Omega(\Delta))$ the same holds for all choices of $b'$ and $m$ \emph{simultaneously}.

We now always assume that this good event occurs. Let $s\le |R|$ be maximal such that 
\begin{equation}\label{eq:defk}
\sum_{i=1}^{s} e(i) < \frac{17}{2}\lambda  \Delta^2.
\end{equation}
We distinguish two cases. If $s=|R|$, then for all rounds $r$ of the game process we have
\begin{equation}\label{eq:goodsit}
L_r = \sum_{u \in \Gamma'_r(v)} \ell_r(u) \le \sum_{i=1}^{r} e(i) < \frac{17}{2}\lambda \Delta^2.
\end{equation}
In the case $s < |R|$ we want to show that \eqref{eq:goodsit} holds at least for all rounds $r$ where $\ell_r(v) < 2\lambda b^{-1}\Delta$.
For all $1 \le b' \le 2b+2$ let $\alpha(b')$ count the number of rounds $r_i \in R$ such that $i \le s$ and $e(i)=b'$. Since $e(s)\le 2b+2$, inequality \eqref{eq:defk} implies
$$
\sum_{b'=1}^{2b+2} b' \cdot \alpha(b') \ge \frac{17}{2}\lambda  \Delta^2-(2b+2).
$$
Therefore,
\begin{equation}\label{eq:alphabound1}
\sum_{i=1}^{s} X(i) \ge \sum_{b'=1}^{2b+2} \sum_{j=1}^{\alpha(b')} Y^{b'}_j \ge \sum_{b'=1}^{2b+2} \Big((1-q)\frac{b' \alpha(b)}{4b\Delta} - \frac{\lambda\Delta}{32b(b+1)}\Big) \ge (1-q) \frac{17\lambda\Delta}{8b}-\frac{1-q}{\Delta} - \frac{\lambda\Delta}{16b}.
\end{equation}
Recall that $q=\frac{6c}{\lambda}$ and note that the $v$-edge that Maker eventually plays at round $r_s+1$ does not contribute to $\ell_{r_s}(v)$. Then for $c$ sufficiently small and $\Delta$ sufficiently large, \eqref{eq:alphabound1} yields
$$\ell_{r_s}(v) \ge \sum_{i=1}^{s}X(i)-1 \ge 2\lambda b^{-1}\Delta.$$
On the other hand, by definition of $s$ for all rounds $r \le r_s$ we have
$$L_r = \sum_{u \in \Gamma'_r(v)} \ell_r(u) \le \sum_{i=1}^{r} \alpha(i) < \frac{17}{2}\lambda \Delta^2.$$

We summarize that in both cases, for all $r$ such that $L_r \ge \frac{17}{2}\lambda  \Delta^2$ we also have $\ell_r(v) \ge 2\lambda b^{-1}\Delta$. However, for all rounds $r$ satisfying $\ell_r(v) < 2\lambda b^{-1}\Delta$, the assumption $\deg(v) \ge (1-c b^{-4})\Delta$ implies
$$\frac{1}{|\Gamma'_r(v)|} \sum_{u \in \Gamma'_r(v)} \ell_r(u) \le \frac{1}{(1-2\lambda b^{-1}-c b^{-4})\Delta} L_r < \frac{17\lambda\Delta}{2(1-2\lambda b^{-1}-c b^{-4})} < 9\lambda  \Delta,$$
given that $\lambda$ and $c$ are sufficiently small. Hence, indeed with probability $1-\exp(-\Omega(\Delta))$ the load $\ell_r(v)$ grows fast enough compared to the average load of the vertices in $\Gamma'_r(v)$.
\end{proof}

\begin{proof}[Proof of Lemma~\ref{lem:colormixing}]
Let $v$ be any fixed vertex and consider the sets $\{I'(u):u \in \Gamma(v)\}$. By definition, the sets $I'(u)$ contain only colors that were chosen by Maker uniformly at random when coloring a good $u$-edge $e$ at a round $r$ such that $\ell_r(u) < \lambda b^{-1}\Delta$. Thus, when Maker is about to color such an edge $e$, the set $U_{r-1}(u)$ has size less than $\lambda b^{-1}\Delta$. Then 
\begin{equation}\label{eq:boundavcolors}
|A_{r-1}(e)| > k-\Delta-\lambda  b^{-1}\Delta = (1-c b^{-4}-\lambda  b^{-1}) \Delta.
\end{equation}

Let $W \subseteq \Gamma(v)$ be a  fixed subset of size $|W|=c b^{-2}\Delta$ and let $i$ be any fixed color. We want to upper bound $\eta_i := |\{u \in W: i \in I'(u)\}|$. Let $u \in W$. Whenever Maker is about to color a good $u$-edge $e$ at a round $r$ and the corresponding color will be added to $I'(u)$, the probability that Maker chooses color $i$ is maximal if $i \in A_{r-1}(e)$ but $|A_{r-1}(e)|$ is as small as possible, i.e. the number of forbidden colors is as large as possible. If $i \notin A_{r-1}(e)$, clearly the probability is zero. Otherwise, by \eqref{eq:boundavcolors}
it is at most 
$$\frac{1}{|A_{r-1}(e)|} \le \frac{1}{(1-cb^{-4}-\lambda b^{-1})\Delta}.$$ 

In order to upper-bound $\eta_i$, we do a worst-case analysis and use a coupling where we assume that whenever Maker colors an edge $e$ and the corresponding color will be contained in $I'(u)$ for some $u \in W$, the probability that $i$ is chosen is precisely $((1-cb^{-4}-\lambda b^{-1})\Delta)^{-1}$. We even allow that neighboring edges get color $i$, which is fine regarding an upper-bound of $\eta_i$. The advantage of this coupling is that the probabilities for Maker choosing color $i$ in such rounds become \emph{independent}, which simplifies the analysis as we get rid of nasty dependencies and case distinctions.
For a single vertex $u \in W$ there are at most $\frac{1}{5b^2}\lambda\Delta$ rounds in which Maker colors a good $u$-edge and adds a color to $I'(u)$. Recall that $|W|=c b^{-2}\Delta$. Hence in total, at most $\frac{1}{5b^4}c\lambda\Delta^2$ rounds have to be considered. It follows that $\eta_i$ is upper-bounded by a random variable $X$ with distribution 
$$X \sim Bin\Big(\frac{c\lambda\Delta^2}{5b^4}, \frac{1}{(1-c b^{-4}-\lambda b^{-1})\Delta}\Big).$$
Clearly, the expected value of $X$ is 
$$\Ex[X] = \frac{c \lambda\Delta}{5b^4(1-c b^{-4}-\lambda b^{-1})} \le \frac{2}{9b^4}c\lambda\Delta,$$ 
and by a Chernoff bound we have
$$\Pr\Big[X \ge \frac{1}{4b^4}c\lambda\Delta\Big] \le \Pr\Big[X \ge \Big(1+\frac18\Big)\Ex[X]\Big] =  \exp(-\Omega(\Delta)).$$
Thus, with probability $1-\exp(-\Omega(\Delta))$ it holds $\eta_i \le \frac{1}{4b^4}c\lambda\Delta$.

Next, let $I^-$ be any set of $c b^{-2}\Delta$ colors. Observe that the random variables $\{\eta_i:i \in I^-\}$ are negatively correlated, because every considered edge can attain at most one color of $I^-$. Hence
$$\Pr\Big[\bigwedge_{i \in I^-} \Big\{\eta_i \ge \frac{1}{4b^4}c\lambda\Delta\Big\}\Big] \le \prod_{i \in I^-} \Pr \Big[\eta_i \ge \frac{1}{4b^4}c\lambda\Delta \Big] \le (\exp(-\Omega( \Delta)))^{|I^-|} \le \exp(-\Omega(\Delta^2)).$$
It remains to union bound over all choices of vertex sets $W$ and color sets $I^-$. We can assume that $\deg(v) \ge c b^{-2}\Delta$, otherwise the statement is trivial. Using the inequality $\binom{n}{k} \le (\frac{ne}{k})^k$ we deduce
\begin{align*}
\Pr\big[\BE_5(v)\big] &\le \binom{\deg(v)}{cb^{-2} \Delta}  \binom{k}{cb^{-2} \Delta}  \exp(-\Omega(\Delta^2))\\ 
&\le \Big(\frac{b^2 e}{c}\Big)^{cb^{-2}\Delta}  \Big(\frac{2b^2 e}{c}\Big)^{cb^{-2}\Delta}\exp(-\Omega(\Delta^2))\\ &= \exp(-\Omega(\Delta^2)).
\end{align*}
\end{proof}

\section{Lower bound for the biased game}\label{sec:lowerbound}

The main idea for proving the lower bound of Theorem~\ref{thm:lowerbiased} is to use a reduction to so-called \emph{Box games}. Box games have been introduced by Chv\'{a}tal and Erd\H{o}s \cite{chvatal1978biased} and are played as follows. There are pairwise disjoint sets $A_1, \ldots, A_s$ such that $|A_i|$ and $|A_j|$ differ by at most $1$ for all choices of $i$ and $j$. Then Alice and Bob take turns (with Alice being the first player) in which they claim previously unclaimed elements of the sets $A_i$. Alice takes one element per round while Bob is allowed to claim up to $b$ elements per turn. Alice wins if she gets at least one element from each set, whereas Bob's goal is to claim all elements of at least one set $A_i$.

Let $f(1,b) := 0$, $f(s,b) := \lfloor \frac{s}{s-1} (f(s-1,b) + b) \rfloor$ for $s \ge 2$. By induction over $s$, we see that for all $s \ge 1$ it holds $f(s,b) \ge (b-1) s \sum_{i=1}^{s-1} \frac{1}{i}$. The following result determines the winner of the Box game.

\begin{theorem}[Theorem~2.1 in \cite{chvatal1978biased}, Corollary~5.4 in \cite{hamidoune1987solution}]\label{thm:boxgame}
Bob has a winning strategy for the Box game if and only if
$$\sum_{i=1}^s |A_i| \le f(s,b).$$
\end{theorem}

Let $b \ge 2$. We want to verify that there are graphs $G=(V,E)$ that attain $\chi_g(G,b) = 2\Delta(G)-1$. A set $F\subseteq E$ is called \emph{good} if (i) for every edge $f \in F$ its two endpoints have degree $\Delta(G)$ in $G$, and (ii) if for all $f_i,f_j \in F$, the distance between $f_i$ and $f_j$ is at least $4$ (that is, either they are in different components of $G$ or every path connecting $f_i$ with $f_j$ has at least three internal nodes). We prove the following statement that is slightly stronger than Theorem~\ref{thm:lowerbiased}.

\begin{lemma}\label{lem:lowerbound}
Let $b \ge 2$ and let $G=(V,E)$ be graph with a good set $F \subseteq E$ such that
\begin{equation}\label{eq:boxcondition}
\frac{2\Delta(G)-2}{b-1} \le \sum_{i=1}^{|F|-1} \frac{1}{i}.
\end{equation}
Then $\chi_g(G,b) = 2\Delta(G)-1$.
\end{lemma}

Let $G=(V,E)$ be a $\Delta$-regular graph with at least $C\Delta^3 \exp(\frac{\Delta-1}{b-1})$ vertices. In $G$, we greedily find a good set $F \subseteq E$ as follows: choose an edge $f$ whose endpoints have both degree $\Delta$, put $f$ into $F$, delete every edge of $E$ with distance at most $2$ to $f$, and iterate as long as possible. Note that whenever we add an edge $f$ to $F$, so far no edge incident to $f$ has been removed because the endpoints of $f$ still have degree $\Delta$. So $f$ has distance at least $4$ to all edges that are already included in $F$, and by induction, $F$ is a good set. Furthermore, whenever an edge $f$ is added to $F$ and edges of $E$ are deleted, we reduce the degree of at most $2\Delta^3$ vertices in $V$. In particular the number of vertices of degree $\Delta$ shrinks by at most $2\Delta^3$ per iteration and we obtain a set $F$ of size at least $\frac{C}{2e^2}\exp(\frac{2\Delta-2}{b-1})$. This implies 
$$\frac{2\Delta-2}{b-1} \le \log|F| - \log \Big(\frac{C}{2e^2}\Big) \le \sum_{i=1}^{|F|-1} \frac{1}{i},$$
where the second inequality follows if $C$ is sufficiently large, no matter how large $|F|$ is. We see that indeed, Theorem~\ref{thm:lowerbiased} is a corollary of Lemma~\ref{lem:lowerbound}.

\begin{proof}[Proof of Lemma~\ref{lem:lowerbound}]
Let $F=\{f_1, \ldots, f_{s}\} \subseteq E$ be a good set of a graph $G$ that satisfies \eqref{eq:boxcondition}, and consider the edge coloring game played with colors $\{1, \ldots, k\}$, where $k < 2\Delta(G)-1$. We want to show that Breaker has a strategy such that at least one edge $f_i \in F$ runs out of available colors before it gets colored. Let $F' := \cup_{i=1}^s \Gamma(f_i)$, where $\Gamma(f_i)$ denotes the set of neighboring edges of $f_i$, i.e., the set of edges that share an endpoint with $f_i$. In the following we suppose that as long as possible, Breaker only colors edges of $F'$. Moreover, we assume that whenever Breaker colors a neighbor of some $f_i \in F$, he uses a color that was so far not used at any neighbor of $f_i$, if possible. 

The reduction from the coloring game to Box games now works as follows: for every edge $f_i \in F$ we introduce a box $A_i$, containing precisely $k$ elements. Whenever Breaker colors an edge of $F'$ that is a neighbor of some edge $f_i \in F$, in the Box game this corresponds to Bob claiming an element of $A_i$. Breaker's right to color at most $b$ edges per turn is mapped to the rule that in the Box game, Bob is allowed to claim up to $b$ elements per turn. Furthermore, whenever Maker colors an edge $e \in E$, we couple this by Alice playing an element of a box $A_i$, where $i$ is chosen such that 
\begin{equation}\label{eq:mapMakeredge}
d_G(f_i,e)=\min_{1 \le j \le s} \{d_G(f_j,e)\}.
\end{equation}
Hence, as long as Breaker colors edges of $F'$, we can interpret the game process as Alice and Bob playing a Box game. Since $F$ fulfills \eqref{eq:boxcondition}, we have
$$\sum_{i=1}^s |A_i| \le s (2\Delta(G)-2) \le s(b-1)\sum_{i=1}^{s-1} \frac{1}{i}  \le f(s,b).$$

By Theorem~\ref{thm:boxgame}, Bob has a winning strategy for this Box game, meaning that he is able to claim all $k$ elements of at least one box $A_i$ before Alice can claim one element of $A_i$. Then our coupling implies that in the coloring game, Breaker has a strategy such that for at least one $f_i \in F$, he can color $k$ neighbors of $f_i$ before Maker colors any edge $e$ fulfilling \eqref{eq:mapMakeredge} for $f_i$. As $E$ is a good set, this means that Maker only colored edges of distance at least $2$ to $f_i$, i.e., he never blocked a color for a neighboring edge of $f_i$. But then, due to the provided strategy, Breaker was able to use all $k$ colors exactly once when coloring the $k$ neighbors of $f_i$. Afterwards, for $f_i$ clearly all colors are forbidden and Breaker wins the edge coloring game with bias $b$ on the graph $G$.
\end{proof}

\section{Open problems}\label{sec:openproblems}

With Theorem~\ref{thm:main} we made a first step towards a proof of Conjecture~\ref{conj:upper}. We verified the statement for all graphs $G$ that satisfy $\Delta(G) \ge C \log v(G)$ by applying a random strategy for Maker. Our attempts to prove the full conjecture were not successful, neither by using the same strategy nor by analyzing more advanced and refined strategies. It is reasonable to believe that from Maker's perspective, the game is harder to win in the case $\Delta(G) \le C \log v(G)$, as indicated by Theorem~\ref{thm:upperbiased} and Theorem~\ref{thm:lowerbiased} for the biased version of the game where the behaviour of $\chi'_g(G,b)$ actually changes around $\Delta(G) \approx \log v(G)$.

In \cite{beveridge2008game} it is also conjectured that there exist $c, d_0>0$ such that every graph $G$ with minimum degree $\delta(G) \ge d_0$ satisfies $\chi'_g(G) \ge (1+c)\Delta(G)$. The interesting case of this statement is when $G$ is almost-regular, i.e., $\Delta(G) \le (1+c)\delta(G)$. Note that so far, this conjecture is not even solved for examples like complete graphs.
Another open question is to decide whether there exist $c, \Delta_0$ such that for any $\Delta \ge \Delta_0$, there are two $\Delta$-regular graphs $G_1$ and $G_2$ with $|\chi'_g(G_1)-\chi'_g(G_2)| \ge c\Delta$. Finally, in order to gain a better understanding of the game process it would be desirable to determine the asymptotic expression of the game chromatic index at least for complete graphs, random graphs, or complete bipartite graphs.

\bibliographystyle{plain}
\bibliography{refs}

\begin{thebibliography}{10}

\bibitem{andres2006game}
Stephan~D. Andres.
\newblock The game chromatic index of forests of maximum degree at least 5.
\newblock {\em Discrete Applied Mathematics}, 154(9):1317--1323, 2006.

\bibitem{bartnicki2008game}
Tomasz Bartnicki and Jaorslaw Grytczuk.
\newblock A note on the game chromatic index of graphs.
\newblock {\em Graphs and Combinatorics}, 24(2):67--70, 2008.

\bibitem{bartnicki2007map}
Tomasz Bartnicki, Jaorslaw Grytczuk, Henry~A. Kierstead, and Xuding Zhu.
\newblock The map-coloring game.
\newblock {\em American Mathematical Monthly}, 114(9):793--803, 2007.

\bibitem{beveridge2008game}
Andrew Beveridge, Tom Bohman, Alan Frieze, and Oleg Pikhurko.
\newblock Game chromatic index of graphs with given restrictions on degrees.
\newblock {\em Theoretical Computer Science}, 407(1-3):242--249, 2008.

\bibitem{bohman2007game}
Tom Bohman, Alan Frieze, and Benny Sudakov.
\newblock The game chromatic number of random graphs.
\newblock {\em Random Structures \& Algorithms}, 32(2):223--235, 2007.

\bibitem{cai2001game}
Leithen Cai and Xuding Zhu.
\newblock Game chromatic index of $k$-degenerate graphs.
\newblock {\em Journal of Graph Theory}, 36:144--155, 2001.

\bibitem{chvatal1978biased}
V\'{a}clav Chv\'{a}tal and Paul Erd\H{o}s.
\newblock Biased positional games.
\newblock {\em Annals of Discrete Math}, 2:221--228, 1978.

\bibitem{erdoes2004note}
Peter~L. Erd\H{o}s, Ulrich Faigle, Winfried Hochst\"{a}ttler, and Walter Kern.
\newblock Note on the game chromatic index of trees.
\newblock {\em Theoretical Computer Science}, 313(3):371--376, 2004.

\bibitem{faigle1993game}
Ulrich Faigle, Walter Kern, Henry~A. Kierstead, and William~T. Trotter.
\newblock On the game chromatic number of some classes of graphs.
\newblock {\em Ars Combinatoria}, 35(17):143--150, 1993.

\bibitem{frieze2013game}
Alan Frieze, Simcha Haber, and Mikhail Lavrov.
\newblock On the game chromatic number of sparse random graphs.
\newblock {\em SIAM Journal on Discrete Mathematics}, 28(2):768--790, 2013.

\bibitem{hamidoune1987solution}
Yahya~O. Hamidoune and Michel~Las Vergnas.
\newblock A solution to the box game.
\newblock {\em Discrete Mathematics}, 65(2):157--171, 1987.

\bibitem{keusch2014game}
Ralph Keusch and Angelika Steger.
\newblock The game chromatic number of dense random graphs.
\newblock {\em The Electronic Journal of Combinatorics}, 21:{\#}P4.47, 2014.

\bibitem{kierstead1994planar}
Henry~A. Kierstead and William~T. Trotter.
\newblock Planar graph coloring with an uncooperative partner.
\newblock {\em Journal of Graph Theory}, 18(6):564--584, 1994.

\bibitem{lam1999edge}
Peter C.~B. Lam, Wai~C. Shu, and Baogang Xu.
\newblock Edge game coloring of graphs.
\newblock {\em Graph Theory Notes N.~Y.}, 37:17--19, 1999.

\bibitem{spencer1991randomization}
Joel Spencer.
\newblock Randomization, derandomization, and antirandomization: Three games.
\newblock {\em Theoretical Computer Science}, 131(2):415--429, 1991.

\bibitem{zhu2008refined}
Xuding Zhu.
\newblock Refined activation strategy for the marking game.
\newblock {\em Journal of Combinatorial Theory, Series B}, 98(1):1--18, 2008.

\end{thebibliography}
\end{document}